\documentclass[11pt, a4paper]{article}

\usepackage[a4paper, top=2.5cm, bottom=2.5cm, left=2.5cm, right=2.5cm]{geometry}

\usepackage[T1]{fontenc}      
\usepackage[utf8]{inputenc}     
\usepackage[english]{babel}     
\usepackage{lmodern}          

\usepackage{amsmath} 
\usepackage{amssymb} 
\usepackage{amsthm}  
\usepackage{xparse}  

\newcommand{\N}{\mathbb{N}} 
\newcommand{\Z}{\mathbb{Z}} 
\newcommand{\Qu}{\textbf{Q}}
\newcommand{\Leb}{\text{Leb}}

\NewDocumentCommand{\dv}{ O{} m m }{ \frac{d\IfNoValueTF{#1}{}{^{#1}}#2}{d#3\IfNoValueTF{#1}{}{^{#1}}} }
\NewDocumentCommand{\pdv}{ O{} m m }{ \frac{\partial\IfNoValueTF{#1}{}{^{#1}}#2}{\partial#3\IfNoValueTF{#1}{}{^{#1}}} }

\newcommand{\abs}[1]{\left\lvert #1 \right\rvert}

\newtheorem{theorem}{Theorem}

\newtheorem{corollary}{Corollary}

\newtheorem{lemma}{Lemma}[section]
\newtheorem{proposition}{Proposition}[section]

\theoremstyle{definition}

\theoremstyle{remark}
\newtheorem{remark}{Remark}[section]

\newtheorem*{question*}{Question}

\usepackage{booktabs} 
\usepackage{hyperref} 
\hypersetup{
    colorlinks=true,
    linkcolor=blue,
    filecolor=magenta,      
    urlcolor=cyan,
    pdftitle={Math Article},
    pdfpagemode=FullScreen,
}

\usepackage[backend=biber,style=alphabetic]{biblatex}
\title{Typical expansive measures for infinite-dimensional maps}
\author{Gabriel Lacerda\thanks{Corresponding author. Email: lacerda@im.ufrj.br} \\ Universidade Federal do Rio de Janeiro}
\date{}

\begin{document}

\maketitle

\begin{abstract}
A classical theorem by Mañé states that compact metric spaces admitting expansive homeomorphisms must be finite-dimensional. It is a natural question, raised by Morales and Sirvent, whether this dimensional constraint extends to the measure theoretic notion of expansiveness. In this paper, we answer this question in the negative. We prove that the shift map on the Hilbert cube is expansive with respect to the product Lebesgue measure. Furthermore, we establish a general sufficient condition: any homeomorphism admitting a full-support expansive measure is densely measure-expansive. As a consequence, we show that the set of expansive and topologically stable measures for the shift map contains a dense $G_\delta$ subset, revealing a rich measure-theoretic structure despite the infinite dimensionality of the phase space.

\vspace{0.5cm}
\end{abstract}

\section{Introduction}

A homeomorphism $f$ defined on a compact metric space $(X, d)$ is called \emph{expansive} if there exists $\delta > 0$ such that for every pair of distinct points $x, y \in X$, there exists $n \in \mathbb{Z}$ such that $d(f^n(x), f^n(y)) > \delta$. Equivalently, $f: X \to X$ is expansive if there exists $\delta > 0$ such that $\Gamma_\delta(x) = \{x\}$ for all $x \in X$, where
\begin{equation*}
    \Gamma_\delta(x) = \{y \in X : d(f^i(x), f^i(y)) \leq \delta, \, \forall i \in \mathbb{Z}\}.
\end{equation*}
In other words, a homeomorphism is expansive if the orbits of distinct points cannot remain within distance $\delta$ for all time. Moreover, we say that $\delta$ is an expansivity constant. This concept was introduced by W. R. Utz in \cite{utz_1950} under the name \emph{unstable homeomorphism} and has been the subject of extensive research in dynamical systems. A classical result by R. Ma\~{n}\'{e} \cite{mane_1979} states that if a compact metric space admits an expansive homeomorphism, it must have finite topological dimension.

Hence, a natural question is whether weaker versions of expansiveness also imply finite topological dimension. This question has already been studied for weaker topological notions of expansiveness. In fact, H. Kato proved in \cite{kato_1993} that if a compact metric space admits a continuum-wise expansive homeomorphism, then it must also have finite topological dimension.

A measure-theoretic notion of expansiveness was introduced by C. A. Morales in \cite{morales_2011} and revisited in a recent preprint \cite{morales_2025}. Intuitively, this notion requires that the measure of the set of points whose orbits remain within distance $\delta$ of a given point is zero. Formally, given a Borel probability measure $\mu$ on $X$, a homeomorphism $f$ is \emph{$\mu$-expansive} if there exists $\delta > 0$ such that $\mu(\Gamma_\delta(x)) = 0$ for all $x \in X$. In this context, we also say that $\mu$ is an \emph{expansive measure} with respect to $f$, or that $f$ is \emph{expansive with respect to $\mu$}, and that $\delta$ is an expansivity constant for $\mu$. Hence, it is not hard to see that every expansive homeomorphism is also expansive with respect to every non-atomic measure.

It is natural to ask whether $\mu$-expansiveness implies that the topological dimension of $X$ is finite. In this work, we answer this question in the negative.

Specifically, consider the Hilbert cube $\Qu = [0, 1]^\Z$, which is a compact metric space when endowed with the Tychonoff topology. A canonical dynamical system on $\Qu$ is the shift map $\sigma: \Qu \to \Qu$. This map is a homeomorphism defined by
\begin{equation*}
    \sigma(\xi) = (\ldots, \xi_0; \xi_1, \xi_2, \ldots)
\end{equation*}
for $\xi = (\ldots, \xi_{-1}; \xi_0, \xi_1, \ldots)$, where $\xi_i \in [0, 1]$. Let $\Leb$ denote the Lebesgue measure on $[0, 1]$. This measure induces a product measure on $\Qu$ defined by $\mathfrak{L} = \Leb^\Z$, which is invariant under $\sigma$. More generally, for a good measure $\nu$ on $[0, 1]$, that is, a measure that is homeomorphic to the Lebesgue measure, denote by $\mathfrak{N}$ the product measure $\mathfrak{N} = \nu^\Z$.

\begin{theorem}\label{theo:shift-map-measure-expansive}
    The shift map $\sigma: \emph{\Qu} \to \emph{\Qu}$ is expansive with respect to $\mathfrak{L}$, and also with respect to the product measure $\mathfrak{N}$. In particular, there exist homeomorphisms of infinite-dimensional compacta exhibiting expansive invariant measures with full support.
\end{theorem}

The above result indicates that the weaker notion of expansiveness does not impose strict topological constraints on the dimension of the phase space, providing a negative answer to a question posed in \cite{morales_sirvent_2013}, namely: \emph{Is the support of an expansive measure of a homeomorphism of a compact metric space finite-dimensional?}

The validity of this theorem can also be derived from Lemma 2.7 in \cite{arbieto_morales_2013}. Although their proof assumes that the metric entropy is positive and finite, one can easily extend their arguments to the case of infinite metric entropy. Consequently, the proof of the first part of Theorem \ref{theo:shift-map-measure-expansive} can be reduced to showing that $\mathfrak{L}$ is an ergodic invariant measure for $\sigma$.

In \cite{lee_morales_shin_2018}, the authors introduced the concept of a densely measure-expansive homeomorphism $f$ on a compact metric space $X$, meaning that the set of expansive measures is dense in the space of Borel probability measures $\mathcal{M}(X)$. They showed that the set of measures expansive with respect to $f$ sharing a common expansivity constant is a $G_\delta$ subset of $\mathcal{M}(X)$ \cite[Lemma 3.1]{lee_morales_shin_2018}, where a $G_\delta$ set is a countable intersection of open sets. Moreover, they established a necessary condition for this property: for every $\varepsilon > 0$, there exists an expansive measure $\mu$ with invariant support such that $d_H(X, \operatorname{supp}(\mu)) \leq \varepsilon$, where $\operatorname{supp}(\mu)$ denotes the support of $\mu$ and $d_H$ is the Hausdorff distance. Recall that if $\operatorname{supp}(\mu) = X$, we say $\mu$ has full support. Here, we provide a sufficient condition.

\begin{theorem}\label{thm:full-support-expansive-measures-imply-densely-measure-expansive}
    If a homeomorphism $f: X \to X$ has an expansive measure with full support, then it is densely measure-expansive. More precisely, the set of expansive measures contains a dense $G_\delta$ subset of $\mathcal{M}(X)$.
\end{theorem}

In the literature, a property is considered \emph{typical} if it holds on a dense $G_\delta$ set. This means that a topologically large set has the property.

Therefore, as an immediate consequence of the above results, since the product Lebesgue measure is a measure with full support on the Hilbert cube, we obtain the following corollary.

\begin{corollary}\label{coro:dense_set_expansive_measures}
    The shift on the Hilbert cube is densely measure-expansive; that is, there exists a set $\mathcal{D}$ of expansive measures for the shift map $\sigma: \emph{\Qu} \to \emph{\Qu}$ which is a dense $G_\delta$ subset of $\mathcal{M}(\emph{\Qu})$ with respect to the weak$^*$ topology.
\end{corollary}

This consequence is surprising because A. Artigue and D. Carrasco-Oliveira showed in \cite{artigue_carrasco-oliveira_2015} that a homeomorphism is expansive with respect to every non-atomic Borel probability measure if and only if it is \emph{countably-expansive}; that is, there exists $\delta > 0$ such that $\Gamma_\delta(x)$ is a countable set for all $x \in X$. Since every countably expansive homeomorphism is continuum-wise expansive, such homeomorphisms can only be defined on finite-dimensional phase spaces. Therefore, our result shows that we are not far from every non-atomic Borel probability measure being expansive, even when the phase space is infinite-dimensional.

We now turn to the discussion of the topological stability of measures. This concept is closely related to the classical notion of topological stability for homeomorphisms. The precise definition of topologically stable measures is given in \cite{lee_morales_2017}. Furthermore, the authors demonstrated that for a homeomorphism satisfying the \emph{pseudo-orbit tracing property} (POTP), all of its expansive measures are also topologically stable. Therefore, to prove the next result, it suffices to show that $\sigma: \Qu \to \Qu$ possesses the POTP.

\begin{corollary}\label{coro:dense_topologically_stable_measures}
    Every measure in $\mathcal{D}$ is also a topologically stable measure for $\sigma$.
\end{corollary}

\section{Dynamical properties of the shift map}

For $\xi, \eta \in \Qu$, consider the metric defined by
\begin{equation}\label{eq:metric_hilbert_cube}
    \rho(\xi, \eta) = \sum_{i \in \mathbb{Z}} \frac{\abs{\xi_i - \eta_i}}{2^{\abs{i}}},
\end{equation}
which induces the Tychonoff topology on the Hilbert cube.

\begin{lemma}\label{lemma:caracterization-epsilon-close-orbit}
Given $\varepsilon > 0$ and $\xi \in \emph{\Qu}$, we have
    \begin{equation}
         \bigotimes_{i \in \mathbb{Z}}\;\; ([\xi_i - \varepsilon/3, \xi_i + \varepsilon/3] \cap [0, 1]) \subset \Gamma_\varepsilon(\xi) \subset \bigotimes_{i \in \mathbb{Z}}\;\; ([\xi_i - \varepsilon, \xi_i + \varepsilon] \cap [0, 1]).
    \end{equation}
\end{lemma}
\begin{proof}
    Consider $\eta \in \bigotimes_{i \in \mathbb{Z}}\;\; ([\xi_i - \varepsilon/3, \xi_i + \varepsilon/3] \cap [0, 1])$, then $\varepsilon/3 \geq \abs{\eta_i - \xi_i}$ for all $i \in \Z$. For any $k \in \Z$, 
    \begin{equation*}
        \sum_{i \in \mathbb{Z}} \frac{\abs{\eta_{k + i} - \xi_{k + i}}}{2^{\abs{i}}} \leq \frac{\varepsilon}{3} \sum_{i \in \mathbb{Z}} \frac{1}{2^{\abs{i}}} = 3 \cdot \frac{\varepsilon}{3} = \varepsilon.
    \end{equation*}
    Hence, for any $k \in \Z$,
    \begin{equation*}
        \rho(\sigma^k(\eta), \sigma^k(\xi)) = \sum_{i \in \mathbb{Z}} \frac{\abs{\eta_{k + i} - \xi_{k + i}}}{2^{\abs{i}}} \leq \varepsilon.
    \end{equation*}
    Thus, $\eta \in \Gamma_\varepsilon(\xi)$. Now, assume that $\eta \in \Gamma_\varepsilon(\xi)$. Then, for every $k \in \Z$, $\rho(\sigma^k(\eta), \sigma^k(\xi)) \leq \varepsilon$. In particular, for every $k \in \Z$, $\abs{\eta_{k + 0} - \xi_{k + 0}} \leq \varepsilon$. In other words, $\abs{\eta_k - \xi_k} \leq \varepsilon$. Therefore, $\eta \in \bigotimes_{i \in \mathbb{Z}}\;\; [\xi_i - \varepsilon, \xi_i + \varepsilon] \cap [0, 1]$.
\end{proof}

\begin{remark}
    The above construction can be generalized to the shift map $\sigma: X^\Z \to X^\Z$, where $(X, d)$ is a compact metric space. Indeed, for $\hat{x} = (\ldots, x_{-1}; x_0, x_1, \ldots)$ in $X^\Z$, modifying the metric defined in (\ref{eq:metric_hilbert_cube}) to account for the metric $d$ yields the same calculation, showing that
    \begin{equation*}
        \bigotimes_{i \in \mathbb{Z}} B_{\varepsilon/3}(x_i) \subset \Gamma_\varepsilon(\hat{x}) \subset \bigotimes_{i \in \mathbb{Z}} B_{\varepsilon}(x_i),
    \end{equation*}
    where $B_\varepsilon(x) = \{y \in X : d(x, y) < \varepsilon\}$. However, when $X$ is a finite set, for sufficiently small $\varepsilon > 0$, $\Gamma_\varepsilon(\hat{x})$ is simply the singleton set $\{\hat{x}\}$, recovering the classical notion of expansiveness.
\end{remark}

We now proceed to prove Theorem \ref{theo:shift-map-measure-expansive}, which establishes that the shift map $\sigma: \Qu \to \Qu$ is expansive with respect to the Lebesgue product measure $\mathfrak{L}$.

\begin{proof}[Proof of Theorem \emph{\ref{theo:shift-map-measure-expansive}}, part 1]
Let $\varepsilon > 0$ satisfy $2\varepsilon < 1$. Then, for every $\xi_i \in [0,1]$, the interval $[\xi_i - \varepsilon, \xi_i + \varepsilon] \cap [0, 1]$ is a proper subset of $[0, 1]$. Moreover, $\Leb([\xi_i - \varepsilon, \xi_i + \varepsilon] \cap [0, 1]) \leq 2\varepsilon$. Hence, for $\xi \in \Qu$, by Lemma \ref{lemma:caracterization-epsilon-close-orbit}, we have
\begin{equation*}
    \mathfrak{L}(\Gamma_\varepsilon(\xi)) \leq \prod\limits_{i \in \Z} \Leb([\xi_i - \varepsilon, \xi_i + \varepsilon] \cap [0, 1]) \leq \prod_{i \in \Z} (2\varepsilon) = 0.
\end{equation*}
This implies that any $\varepsilon > 0$ satisfying $2\varepsilon < 1$ is an expansivity constant for the product Lebesgue measure $\mathfrak{L}$.
\end{proof}

To conclude the proof of Theorem \ref{theo:shift-map-measure-expansive} concerning expansiveness with respect to the product of a good measure on the interval, we first define this concept. For context, a good measure is also known as an \emph{OU measure}, named in honor of J. Oxtoby and S. Ulam. Specifically, a probability measure $\nu$ on $[0,1]$ is called a \emph{good measure} if it is of the form $\nu = \Leb \circ \varphi$, where $\varphi$ is a homeomorphism of $[0, 1]$. That is, for any Borel set $A \subset [0, 1]$, we have $\nu(A) = \Leb(\varphi(A))$. We now proceed to show that $\mathfrak{N} = \nu^\Z$ is also an expansive measure for the shift map $\sigma: \Qu \to \Qu$.

\begin{proof}[Proof of Theorem \emph{\ref{theo:shift-map-measure-expansive}}, part 2]
Let $\varphi$ be a homeomorphism of the interval $[0, 1]$, and let $0 < \varepsilon < 1$. Because $\varphi$ is a continuous function on a compact set, it is uniformly continuous. Therefore, there exists a $\delta > 0$ such that for all $x, y \in [0, 1]$, if $\abs{x - y} \leq \delta$, then $\abs{\varphi(x) - \varphi(y)} < \varepsilon$. Thus, for any given $z \in [0, 1]$,
\begin{equation*}
    \Leb(\varphi([z - \delta/2, z + \delta/2] \cap [0, 1])) < \varepsilon.
\end{equation*}
In other words, $\nu\left(\left[z - \delta/2, z + \delta/2\right] \cap [0, 1]\right) < \varepsilon$. Consequently, for any $\xi \in \Qu$, Lemma \ref{lemma:caracterization-epsilon-close-orbit} implies that
\begin{equation*}
     \Gamma_{\delta/2}(\xi) \subset \bigotimes_{i \in \mathbb{Z}} \left([\xi_i - \delta/2, \xi_i + \delta/2] \cap [0, 1]\right).
\end{equation*}
Therefore, since $0 < \varepsilon < 1$ and $\mathfrak{N}$ is the product measure, we obtain
\begin{equation*}
    \mathfrak{N}( \Gamma_{\delta/2}(\xi)) \leq \prod_{i \in \mathbb{Z}} \nu\left([\xi_i - \delta/2, \xi_i + \delta/2] \cap [0, 1]\right) \leq \prod_{i \in \mathbb{Z}} \varepsilon = 0.
\end{equation*}
\end{proof}

\begin{remark}
    Observe that the above argument is also valid for the product measure $\mathfrak{N}_* = \bigotimes_{i \in \Z} \nu_i$ on $\Qu$, where each $\nu_i = \Leb \circ \varphi_i$ is a good measure and the associated set of homeomorphisms $\{\varphi_i\}$ forms an equicontinuous family. Hence, this shows the existence of a measure that is not necessarily invariant, but is also an expansive measure with respect to $\sigma$.
\end{remark}

To metrize the weak-$*$ topology on $\mathcal{M}(X)$, we consider the Prokhorov metric $\pi$. Given the Borel $\sigma$-algebra $\mathcal{B}_X$ on $X$, the distance between two measures $\mu, \nu \in \mathcal{M}(X)$ is defined as
\begin{equation}\label{eq:prohorov-metric-definition}
    \pi(\mu, \nu) = \inf \{\delta > 0 : \mu(A) \leq \nu(A_\delta) + \delta \text{ for all } A \in \mathcal{B}_X\},
\end{equation}
where $A_\delta = \bigcup_{x \in A} \{y \in X : d(x, y) < \delta\}$ \cite{Prokhorov_1956}. 

\begin{proof}[Proof of Theorem \ref{thm:full-support-expansive-measures-imply-densely-measure-expansive}]
    The proof of this result relies on the fact that, given a measure $\mu$ with full support, the set of non-atomic absolutely continuous measures with respect to $\mu$ is dense. We provide the explicit construction here. 
    
    For every $\varepsilon > 0$ and any measure $\lambda \in \mathcal{M}(X)$, there exists a measure $\beta$ with finite support such that $\pi(\lambda, \beta) < \varepsilon/2$. We can write $\beta = \sum_{i = 1}^N t_i \delta_{x_i}$, where $\delta_{x_i}$ is the Dirac measure at $x_i \in X$, $t_i > 0$, and $\sum_{i = 1}^N t_i = 1$. 
    
    We will construct a non-atomic measure $\nu$ such that $f$ is $\nu$-expansive and $\pi(\beta, \nu) < \varepsilon/2$. For each $x_i \in \operatorname{supp}(\beta)$, choose a small Borel set $F_i$ containing $x_i$ such that $\operatorname{diam}(F_i) < \varepsilon/2$ and $\mu(F_i) > 0$. This is possible because $\mu$ has full support. Next, for each $1 \leq i \leq N$, define the probability measure $\nu_i$ such that for any measurable set $A$,
    $$
        \nu_i(A) = \frac{\mu(A \cap F_i)}{\mu(F_i)}.
    $$
    We then define the measure $\nu$ as $\nu = \sum_{i = 1}^N t_i\nu_i$. 
    
    Since $f$ is $\mu$-expansive, there exists $\delta > 0$ such that $\mu(\Gamma_\delta(x)) = 0$ for every $x \in X$. Because $\nu$ is absolutely continuous with respect to $\mu$, it follows that $\nu(\Gamma_\delta(x)) = 0$, implying that $f$ is also $\nu$-expansive. Furthermore, since $\mu$ is non-atomic, $\nu$ is also a non-atomic probability measure. Finally, observe that $\nu$ has the same expansivity constant as $\mu$.
    
    To bound the Lévy–Prokhorov distance between $\beta$ and $\nu$, let $A \subseteq X$ be an arbitrary Borel set, and let $I \subseteq \{1, \ldots, N\}$ be the set of indices such that $x_j \in A$. Since $\operatorname{diam}(F_j) < \varepsilon/2$ and $x_j \in A$, it follows that $F_j \subseteq A_{\varepsilon/2}$ for all $j \in I$. Consequently, $\nu_j(A_{\varepsilon/2}) = 1$ for all $j \in I$. Thus, we have:
    $$
        \beta(A) = \sum_{j \in I} t_j = \sum_{j \in I} t_j\nu_j(A_{\varepsilon/2}) \leq \sum_{j = 1}^N t_j\nu_j(A_{\varepsilon/2}) = \nu(A_{\varepsilon/2}) \leq \nu(A_{\varepsilon/2}) + \frac{\varepsilon}{2}.
    $$
    Therefore, $\pi(\beta, \nu) \leq \varepsilon/2$. By the triangle inequality, $\pi(\lambda, \nu) < \varepsilon$. Finally, there exists a dense set of expansive measures for $f$ that all share the same expansivity constant; hence, by Lemma 3.1 in \cite{lee_morales_shin_2018}, the set of expansive measures contains a dense $G_\delta$ subset of $\mathcal{M}(X)$.
\end{proof}

The proof of Corollary \ref{coro:dense_set_expansive_measures} follows directly from Theorems \ref{theo:shift-map-measure-expansive} and \ref{thm:full-support-expansive-measures-imply-densely-measure-expansive}. That is, it suffices to show that the support of $\mathfrak{L}$ is the entire Hilbert cube $\Qu$.

\begin{proof}[Proof of Corollary \emph{\ref{coro:dense_set_expansive_measures}}]
    It suffices to show that every open neighborhood of a point $\xi \in \Qu$ has positive product Lebesgue measure. Because we are considering the natural topology on $\Qu$, that is, the Tychonoff topology, a basic open neighborhood of $\xi = (\ldots, \xi_{-1}; \xi_0, \xi_1, \ldots)$ is a set of the form $\bigotimes_{i \in \Z} V_i$, where $V_i$ is an open neighborhood of $\xi_i$ for $i$ in a finite subset $F \subset \Z$, and $V_i = [0, 1]$ for all $i \in \Z \setminus F$. Therefore,
\begin{equation*}
    \mathfrak{L}\left(\bigotimes_{i \in \Z} V_i\right) = \prod_{i \in \Z} \Leb(V_i) = \prod_{i \in F} \Leb(V_i) > 0. 
\end{equation*}
\end{proof}

To prove Corollary \ref{coro:dense_topologically_stable_measures}, it suffices to show that $\sigma: \Qu \to \Qu$ satisfies the pseudo-orbit tracing property. Recall that given a homeomorphism $f: X \to X$ and $\delta > 0$, a $\delta$-pseudo-orbit of $f$ is a sequence $(x_n)_{n \in \Z}$ in $X$ such that $d(f(x_n), x_{n + 1}) \leq \delta$ for all $n \in \Z$. Moreover, a $\delta$-pseudo-orbit of $f$ is $\delta$-shadowed if there exists a point $y \in X$ such that $d(f^n(y), x_n) \leq \delta$ for all $n \in \Z$. We say that $f$ satisfies the pseudo-orbit tracing property if for every $\varepsilon > 0$, there exists a $\delta > 0$ such that every $\delta$-pseudo-orbit of $f$ is $\varepsilon$-shadowed.

\begin{proposition}
    The shift map $\sigma: \emph{\Qu} \to \emph{\Qu}$ has the pseudo-orbit tracing property.
\end{proposition}
\begin{proof}
Given $\varepsilon > 0$, let $(\xi^{(k)})_{k \in \Z}$ be a $\delta$-pseudo-orbit of $\sigma$ for some $\delta > 0$. Define $\eta \in \Qu$ by $\eta_k = \xi^{(k)}_0$. We will show that $(\xi^{(k)})_{k \in \Z}$ is $\varepsilon$-shadowed by $\eta$. Indeed, observe that
\begin{equation*}
    \rho(\sigma^k(\eta), \xi^{(k)}) = \sum\limits_{i \in \Z} \frac{\abs{\eta_{i+k} - \xi^{(k)}_i}}{2^{\abs{i}}} = \sum\limits_{i \in \Z} \frac{\abs{\xi^{(i + k)}_0 - \xi^{(k)}_i}}{2^{\abs{i}}}.
\end{equation*}
Since $(\xi^{(k)})_{k \in \Z}$ is a $\delta$-pseudo-orbit, $\rho(\sigma(\xi^{(k)}), \xi^{(k + 1)}) \leq \delta$. Hence, in particular,
\begin{equation*}
    \frac{\abs{\xi^{(k)}_{i + 1} - \xi^{(k + 1)}_i}}{2^{\abs{i}}} \leq \delta.
\end{equation*}
That is, $\abs{\xi^{(k)}_{i + 1} - \xi^{(k + 1)}_i} \leq \delta\cdot 2^{\abs{i}}$ for every $i$ and $k$ in $\Z$. Assume, without loss of generality, that $i > 0$. Applying the triangle inequality, we obtain
\begin{equation*}
    \abs{\xi^{(i + k)}_0 - \xi^{(k)}_i} \leq \sum\limits_{j = 0}^{i-1} \abs{\xi^{(k + i - j)}_j - \xi^{(k + i - j - 1)}_{j + 1}} \leq \sum\limits_{j = 0}^{i-1} \delta\cdot 2^j = \delta(2^i - 1).
\end{equation*}
In general, $\abs{\xi^{(i + k)}_0 - \xi^{(k)}_i} < \delta\cdot2^{\abs{i}}$. Furthermore, choose $N \in \N$ such that
\begin{equation*}
    \sum\limits_{\abs{i} > N} \frac{1}{2^{\abs{i}}} < \frac{\varepsilon}{2},
\end{equation*}
and choose $\delta > 0$ such that $\delta < \frac{\varepsilon}{2(2N + 1)}$. Therefore,
\begin{equation*}
    \rho(\sigma^k(\eta), \xi^{(k)}) = \sum\limits_{i \in \Z} \frac{\abs{\xi^{(i + k)}_0 - \xi^{(k)}_i}}{2^{\abs{i}}} \leq \sum\limits_{\abs{i} \leq N} \frac{\abs{\xi^{(i + k)}_0 - \xi^{(k)}_i}}{2^{\abs{i}}}  + \sum\limits_{\abs{i} > N} \frac{1}{2^{\abs{i}}} \leq \left(\sum\limits_{\abs{i} \leq N} \delta\right) + \frac{\varepsilon}{2}.
\end{equation*}
Observe that $\sum_{\abs{i} \leq N} \delta = (2N + 1)\delta$. Hence, by the choice of $\delta$, $\rho(\sigma^k(\eta), \xi^{(k)}) < \varepsilon$ for all $k \in \Z$.
\end{proof}

\begin{remark}
    A computation similar to the one in the preceding proof can be made for the shift map $\sigma: X^\Z \to X^\Z$. Indeed, for any $\varepsilon > 0$, there exists a constant $C > 0$ depending on the metric $d$ on $X$, such that if $(\hat{x}^{(k)})_{k \in \Z}$ is a $(C\delta)$-pseudo-orbit of $\sigma$, with $\delta > 0$ chosen as in the proof above, then the point $\hat{y} = (\ldots, y_{-1}, y_0, y_1, \ldots) \in X^\Z$ defined by $y_k = x^{(k)}_0$ $\varepsilon$-shadows $(\hat{x}^{(k)})_{k \in \Z}$.
\end{remark}

Given $B \subset X$, a sequence $(x_n)_{n \in \mathbb{Z}}$ is said to be \emph{through} $B$ if $x_0 \in B$. A Borel measure $\mu$ has the \emph{POTP} with respect to $f$ if for every $\varepsilon > 0$, there exist $\delta > 0$ and a Borel set $B \subset X$ with $\mu(X \setminus B) = 0$ such that every $\delta$-pseudo-orbit through $B$ can be $\varepsilon$-shadowed. In other words, this is the pseudo-orbit tracing property in measure. Observe that if a homeomorphism $f: X \to X$ has the POTP, then every Borel measure $\mu$ on $X$ also has the POTP.

Moreover, we say that a Borel measure $\mu$ is \emph{topologically stable} with respect to $f$ if for every $\varepsilon > 0$, there exists $\delta > 0$ such that for every homeomorphism $g$ that is $\delta$-close to $f$, there is an upper semi-continuous compact-valued map $H$ on $X$ with a measurable domain $\operatorname{Dom}(H) = \{x \in X : H(x) \neq \varnothing\}$ such that:
\begin{enumerate}
    \item $\mu(X\setminus\operatorname{Dom}(H)) = 0$;
    \item $\mu(H(x)) = 0$ for all $x \in X$;
    \item $H(x) \subset \{y \in X : d(x, y) \leq \varepsilon\}$ for every $x \in X$;
    \item $f\circ H = H\circ g$.
\end{enumerate}

We are now ready to prove that every expansive measure in $\mathcal{D}$, as established in Corollary \ref{coro:dense_set_expansive_measures}, is also topologically stable.

\begin{proof}[Proof of Corollary \emph{\ref{coro:dense_topologically_stable_measures}}]
    Because $\sigma: \Qu \to \Qu$ satisfies the pseudo-orbit tracing property, every Borel measure $\mu \in \mathcal{D}$ satisfies the POTP. Furthermore, since every expansive measure satisfying the POTP is topologically stable by Theorem 3.1 of \cite{lee_morales_2017}, it follows that $\mu$ is topologically stable.
\end{proof}

\subsection{An alternate proof of Theorem \ref{theo:shift-map-measure-expansive}}

In the following result, we extend Lemma 2.7 from \cite{arbieto_morales_2013} to invariant measures with infinite entropy. Although our proof is almost identical to theirs, their argument relies on the fact that the entropy of the invariant measure is positive and finite.

Recall that a Borel probability measure $\mu$ is called \emph{ergodic} if every measurable invariant set of $f$ has measure zero or one. The entropy of $\mu$ with respect to $f$ is defined by
\begin{equation*}
    h_\mu(f) = \sup \left\{-\lim\limits_{n \to \infty} \frac{1}{n}\sum_{P \in \mathcal{P}_{n - 1}} \mu(P)\log(\mu(P)) : \mathcal{P} \text{ is a finite partition of } X\right\},
\end{equation*}
where $\mathcal{P}_n$ is the pullback partition of $\mathcal{P}$ under $f^n$ (see \cite{Viana_Oliveira_2016} for details).

\begin{lemma}\label{lemma:ergodic-invariant-measure-infinite-entropy-is-expansive}
    For any homeomorphism on a compact metric space, every ergodic invariant measure with infinite entropy is expansive.
\end{lemma}
\begin{proof}
    Given $\mu$ an infinite entropy invariant measure of $f: X \to X$, fix $\delta > 0$ and consider $X_\delta = \{x \in X : \mu(\Gamma_{\delta}(x)) = 0\}$. By Lemma 1.20 in \cite{morales_sirvent_2013}, we only need to show that there is $\delta > 0$ such that $\mu(X_\delta) = 1$.

    Fix $x \in X$ and let $y \in f(\Gamma_{\delta}(x))$. Then there exists $y_0 \in \Gamma_{\delta}(x)$ such that $y = f(y_0)$. In particular, this implies $d(y, f(x)) \leq \delta$. Since $y = f(y_0)$ and $y_0 \in \Gamma_{\delta}(x)$, we have $y \in \Gamma_{\delta}(f(x))$. Hence, $f(\Gamma_{\delta}(x)) \subset \Gamma_{\delta}(f(x))$. In other words,
    \begin{equation*}
        \Gamma_{\delta}(x) \subset f^{-1}(\Gamma_{\delta}(f(x))).
    \end{equation*}
    Since $\mu$ is an invariant measure for $f$, we obtain
    \begin{equation*}
        \mu(\Gamma_{\delta}(x)) \leq \mu(\Gamma_{\delta}(f(x))).
    \end{equation*}
    Consequently, $x$ belongs to $X_\delta$ whenever $f(x) \in X_\delta$. That is, $f(X_\delta) \subset X_\delta$. Equivalently, $X_\delta \subset f^{-1}(X_\delta)$. Note that $\mu(X_\delta) = \mu(f^{-1}(X_\delta))$. Since $X_\delta$ is contained in $f^{-1}(X_\delta)$, we have $\mu(X_\delta \Delta f^{-1}(X_\delta)) = 0$, where $X_\delta \Delta f^{-1}(X_\delta)$ is the symmetric difference between $X_\delta$ and $f^{-1}(X_\delta)$. Since $\mu$ is an ergodic measure, then for every $\delta > 0$, $\mu(X_\delta) \in \{0, 1\}$. Hence, we only need to show that there exists $\delta > 0$ such that $\mu(X_\delta) > 0$.

    For every $\delta > 0$, define $\phi_\delta: X \to \mathbb{R}\cup\{\infty\}$ as
    \begin{equation*}
        \phi_\delta(x) = \liminf\limits_{n \to \infty} - \frac{\log \mu(B[x, n, \delta])}{2n - 1},
    \end{equation*}
    where $B[x, n, \delta] = \{y \in X : d(f^i(x), f^i(y)) \leq \delta, \text{ for } -n + 1 \leq i \leq n - 1\}$. For $m, L \in \mathbb{N}$, set $X^{m, L} = \{x \in X : \phi_{1/m}(x) > L\}$. For a fixed $L$, $X^{m, L} \subset X^{m', L}$, whenever $m \leq m'$. Hence,
    \begin{equation*}
        \{x \in X : \sup\limits_{\delta > 0} \phi_\delta(x) = \infty\} \subset \bigcup_{m \in \mathbb{N}} X^{m, L}.
    \end{equation*}
    Thus,
    \begin{equation*}
        \mu(\{x \in X : \sup\limits_{\delta > 0} \phi_\delta(x) = \infty\}) \leq \lim\limits_{m \to \infty} \mu(X^{m, L}).
    \end{equation*}
    Since $\mu$ is an ergodic invariant measure with infinite entropy, $\mu$ is non-atomic. So, the Brin--Katok theorem \cite{brin_katok_1983} implies $\mu(\{x \in X : \sup_{\delta > 0} \phi_\delta(x) = \infty\}) = 1$. Therefore, for any $L \in \mathbb{N}$, $\lim_{m \to \infty} \mu(X^{m, L}) = 1$. This implies that there exists $m_0 \in \mathbb{N}$ such that for all $m \geq m_0$, $\mu(X^{m, L}) > 0$. For $x \in X^{m, L}$, $\mu(B[x, n, 1/m]) < e^{-L(2n - 1)}$ for $n$ sufficiently large. Since $L > 0$, we get 
    \begin{equation*}
        \lim\limits_{n \to \infty} \mu(B[x, n, 1/m]) = 0.
    \end{equation*}
    Moreover, since $\Gamma_{1/m}(x) = \bigcap\limits_{n = 0}^\infty B[x, n, 1/m]$ and $B[x, n, 1/m] \subset B[x, n', 1/m]$ whenever $n \geq n'$, then
    \begin{equation*}
        \mu(\Gamma_{1/m}(x)) = \lim\limits_{n \to \infty} \mu(B[x, n, 1/m]) = 0.
    \end{equation*}
    Hence, $x \in X_{1/m}$. Therefore, for $\delta = 1/m$, we obtain $X^{m, L} \subset X_\delta$. Which implies $\mu(X_\delta) \geq \mu(X^{m, L}) > 0$.
\end{proof}

\begin{remark}\label{remark:product-lebesgue-measure-infinite-entropy}
    The proof that the entropy of $\mathfrak{L}$ with respect to the shift map $\sigma$ is infinite is quite elementary, so we omit it here. The main idea is to show that, for any $m \geq 2$, the full shift on $m$ symbols $\sigma_m$ is a measurable factor of $\sigma: \Qu \to \Qu$. Indeed, consider the intervals $I_i = [i/m, (i+1)/m)$ for $0 \leq i \leq m - 1$, and define $\psi_m: \Qu \to \{0, \dots, m-1\}^\mathbb{Z}$ by $(\psi_m(\xi))_n = i$ if and only if $\xi_n \in I_i$. Clearly, $\psi_m\circ\sigma = \sigma_m\circ\psi_m$, and $(\psi_m)_*\mathfrak{L}$ is the Bernoulli measure on $\{0, \dots, m-1\}^\mathbb{Z}$. Therefore, for each $m \geq 2$, 
    \begin{equation*}
        h_{\mathfrak{L}}(\sigma) \geq h_{(\psi_m)_*\mathfrak{L}}(\sigma_m) = \log m.
    \end{equation*}
\end{remark}

\begin{lemma}\label{lemma:product-lebesgue-measure-ergodic}
    The product Lebesgue measure $\mathfrak{L}$ is an ergodic measure for the shift $\sigma: \Qu \to \Qu$.
\end{lemma}
\begin{proof}
    Let $E \subset \Qu$ be a Borel set such that $\sigma^{-1}(E) = E$. We must show that $\mathfrak{L}(E) \in \{0, 1\}$. Consider a sequence of cylinder sets $C_n$ such that $\mathfrak{L}(E\Delta C_n) \to 0$. This implies that $\mathfrak{L}(C_n) \to \mathfrak{L}(E)$. Moreover, since each cylinder set $C_n$ is defined by coordinates with indices in $[-n, n]$, we can choose an increasing sequence of integers $(k_n)_{n \in \mathbb{N}}$ such that $k_n > 2n$. Because the coordinates are disjoint, the sets are independent under the product measure $\mathfrak{L}$, and by shift-invariance we obtain, for each $n \in \mathbb{N}$,
    \begin{equation*}
        \mathfrak{L}(C_n \cap \sigma^{-k_n}(C_n)) = \mathfrak{L}(C_n)\mathfrak{L}(\sigma^{-k_n}(C_n)) = \mathfrak{L}(C_n)^2.
    \end{equation*}
    Furthermore, observe that
    \begin{equation*}
        \mathfrak{L}(E \Delta \sigma^{-k_n}(C_n)) = \mathfrak{L}(\sigma^{-k_n}(E) \Delta \sigma^{-k_n}(C_n)) = \mathfrak{L}(E \Delta C_n),
    \end{equation*}
    by the invariance of $E$ and of $\mathfrak{L}$, which implies that $\mathfrak{L}(E\Delta\sigma^{-k_n}(C_n)) \to 0$. Thus, 
    \begin{equation*}
        \mathfrak{L}((C_n \cap \sigma^{-k_n}(C_n)) \Delta E) \to 0.
    \end{equation*}
    This implies that
    \begin{equation*}
        \mathfrak{L}(E) = \lim\limits_{n \to \infty} \mathfrak{L}(C_n \cap \sigma^{-k_n}(C_n)) = \lim\limits_{n \to \infty} \mathfrak{L}(C_n)^2 = \mathfrak{L}(E)^2.
    \end{equation*}
    Since the product Lebesgue measure $\mathfrak{L}$ is a probability measure, this forces $\mathfrak{L}(E) \in \{0, 1\}$, confirming that $\mathfrak{L}$ is an ergodic measure for the shift map.
\end{proof}

\begin{proof}[Alternate proof of Theorem \emph{\ref{theo:shift-map-measure-expansive}}, part 1]
    As a consequence of Remark \ref{remark:product-lebesgue-measure-infinite-entropy} and Lemma \ref{lemma:product-lebesgue-measure-ergodic}, the product Lebesgue measure $\mathfrak{L}$ is an ergodic measure with infinite entropy. Therefore, by Lemma \ref{lemma:ergodic-invariant-measure-infinite-entropy-is-expansive}, we conclude that $\mathfrak{L}$ is an expansive measure for the shift map $\sigma$.
\end{proof}

\subsection{On a generalization of entropy-expansiveness}

In this section, we introduce a generalization of the notion of entropy-expansive maps, originally defined by R. Bowen in \cite{bowen_1972}. This new concept is based on the metric mean dimension developed by E. Lindenstrauss and B. Weiss in \cite{lindenstrauss_weiss_2000}. We will refer to maps satisfying this property as \emph{mmd-expansive}.

Given $\varepsilon > 0$, $n \in \mathbb{N}$, and a closed set $K \subset X$, let $\operatorname{Sep}(f, K, n, \varepsilon)$ denote the maximum number of $\varepsilon$-separated points contained in $K$ with respect to the metric $d_n$, defined for $x, y \in X$ by $d_n(x, y) = \max \{d(f^ix, f^iy) : 0 \leq i \leq n - 1\}$. We define
\begin{equation*}
    \operatorname{Sep}(f, K, \varepsilon) = \limsup\limits_{n \to \infty} \frac{\log \operatorname{Sep}(f, K, n, \varepsilon)}{n},
\end{equation*}
and $h(f, K) = \lim\limits_{\varepsilon \to 0} \operatorname{Sep}(f, K, \varepsilon)$. Finally, the map $f: X \to X$ is said to be \emph{entropy-expansive} if there exists $\delta > 0$ such that $h(f, \Gamma_\delta(x)) = 0$ for every $x \in X$. Moreover, the upper metric mean dimension and the lower metric mean dimension are, respectively, defined as
\begin{equation*}
    \overline{\operatorname{mdim}}_M(f, K, d) = \limsup\limits_{\varepsilon \to 0 } \frac{\operatorname{Sep}(f, K, \varepsilon)}{\abs{\log \varepsilon}} \text{ and } \underline{\operatorname{mdim}}_M(f, K, d) = \liminf\limits_{\varepsilon \to 0 } \frac{\operatorname{Sep}(f, K, \varepsilon)}{\abs{\log \varepsilon}}.
\end{equation*}
Therefore, we say that $f: X \to X$ is \emph{upper mmd-expansive} if there is $\delta > 0$ such that $\overline{\operatorname{mdim}}_M(f, \Gamma_\delta(x), d) = 0$ for all $x \in X$. On the other hand, $f$ is \emph{lower mmd-expansive} if there exists $\delta > 0$ such that $\underline{\operatorname{mdim}}_M(f, \Gamma_\delta(x), d) = 0$ for every $x \in X$. Clearly, if a map is upper mmd-expansive, it is also lower mmd-expansive; therefore, we will simply refer to it as \emph{mmd-expansive}. Furthermore, if a map is entropy-expansive, then it is mmd-expansive.

\begin{question*}
    Is there a map that is lower mmd-expansive but not mmd-expansive?
\end{question*}

\begin{proposition}
    The shift map $\sigma: \emph{\Qu} \to \emph{\Qu}$ is not lower mmd-expansive.
\end{proposition}
\begin{proof}
    Let $\xi_*$ denote the point in the Hilbert cube whose coordinates are all zero. That is, $\xi_* = (\ldots, 0; 0, 0, \ldots)$. Given $\varepsilon > 0$, let $\eta \in [0, \varepsilon/3]^\mathbb{Z}$ and $k \in \mathbb{Z}$. We have that $\sigma^k(\eta) = (\ldots, \eta_{k - 1}; \eta_k, \eta_{k + 1}, \ldots)$. Note that $\sigma^k(\xi_*) = \xi_*$ for all $k \in \mathbb{Z}$. Hence,
    \begin{equation*}
        \rho(\sigma^k(\eta), \xi_*) = \sum_{i \in \mathbb{Z}} \frac{\abs{\eta_{k + i}}}{2^{\abs{i}}} \leq \frac{\varepsilon}{3} \sum_{i \in \mathbb{Z}} \frac{1}{2^{\abs{i}}} = 3 \cdot \frac{\varepsilon}{3} = \varepsilon.
    \end{equation*}
    Thus, $\eta \in \Gamma_\varepsilon(\xi_*)$, which means $[0, \varepsilon/3]^\mathbb{Z} \subset \Gamma_\varepsilon(\xi_*)$. Note that $[0, \varepsilon/3]^\mathbb{Z}$ is invariant by the shift map. Thus, by Proposition 3.3 and Theorem 4.2 in \cite{lindenstrauss_weiss_2000} implies that $\underline{\operatorname{mdim}}_M(\sigma, [0, \varepsilon/3]^\Z, \rho) \geq 1$. Since this holds for all sufficiently small $\varepsilon > 0$, the shift map is not lower mmd-expansive.
\end{proof}

\begin{question*}
    Is there an infinite-dimensional subshift of $\sigma: \Qu \to \Qu$ that is mmd-expansive?
\end{question*}


\section*{Acknowledgments}
The author would like to thank Diego Sanhueza for fruitful discussions and valuable exchange of ideas.

\section*{Funding details}
This work was supported by the Conselho Nacional de Desenvolvimento Científico e Tecnológico (CNPq) through a doctoral scholarship.

\printbibliography

\end{document}